\documentclass[11pt]{extarticle}
\usepackage{enumerate}
\usepackage[spanish,english,activeacute]{babel}
\usepackage{amsmath}
\usepackage{amsfonts}
\usepackage{amssymb,times}
\usepackage{graphicx}
\usepackage{latexsym}
\usepackage{color}
\usepackage{pifont}
\usepackage{multicol}
\usepackage{appendix}
\usepackage{mathrsfs}
\usepackage[sort&compress]{natbib}
\newcommand{\subo}{\mathcal{S}}
\newcommand{\x}{\mathcal{X}}

\newcommand{\pad}{\right)}
\newcommand{\mayor}{>}
\newcommand{\menor}{<}
\newcommand{\EL}{\mathcal{L}}
\newcommand{\pai}{\left(}
\newcommand{\cd}{\right]}
\newcommand{\ci}{\left[}

\newcommand{\colachi}{\overline{\chi}}
\newcommand{\e}{\mathcal{E}}

\newcommand{\constC}{\prod_{i=1}^N\alpha_i^{n_i}\pai\prod_{j=1}^{m+1}\beta_j(q)\pad^{-1}}

\usepackage{hyperref}
\begin{document}
\title{$\,$\vskip-2cm\bf The distribution and asympotic behaviour of the negative Wiener-Hopf factor for L\'evy processes with rational positive jumps}
\author{ {\sc Ekaterina T. Kolkovska\/}\thanks{\'{A}rea de Probabilidad y Estad\'{\i}stica, Centro de Investigaci\'{o}n en Matem\'{a}ticas,
  Guanajuato,  Mexico.} \and
  {\sc Ehyter M. Mart\'{\i}n-Gonz\'{a}lez}\thanks{Departamento de Matem\'aticas, Universidad de Guanajuato,
  Guanajuato,  Mexico.} 
}
\date{ }
\maketitle
\vskip-1cm
\begin{abstract}
  We study the distribution of the negative Wiener-Hopf factor for a class of two-sided jumps L\'evy processes whose positive
  jumps have a rational Laplace transform. The positive Wiener-Hopf factor for this class of processes was studied by \cite{lewismordecki}.
  Here we obtain a formula for  the Laplace transform of the negative Wiener-Hopf factor, as well as an
   explicit expression for its probability density, which is in terms of sums of convolutions of known functions. Under
  additional regularity conditions on the L\'evy measure of the studied  processes, we also provide asymptotic results as $u\to-\infty$ for the distribution function $F(u)$
of the  negative Wiener-Hopf factor.
  We illustrate our results in some particular examples.
    \bigskip

\end{abstract}

\newtheorem{cond}{Condition}
\newtheorem{lemma}{Lemma}
\newtheorem{remark}{Remark}
\newtheorem{propo}{Proposition}
\newtheorem{teo}{Theorem}
\newtheorem{coro}{Corollary}
\newtheorem{defi}{Definition}
\newenvironment{proofofmainteo1}[1][Proof of Theorem \ref{laplacefactorWienerHopfnegativo}]{\textbf{#1.} }{\ \rule{0.5em}{0.5em}}
\newenvironment{proofofmainteo2}[1][Proof of Theorem \ref{densidadfactornegativo}]{\textbf{#1.} }{\ \rule{0.5em}{0.5em}}
\newenvironment{proofoflemaJ0}[1][Proof of Lemma \ref{lemaJ0}]{\textbf{#1.} }{\ \rule{0.5em}{0.5em}}
\newenvironment{proofoflemmatranslationtail}[1][Proof of Lemma \ref{lemmatranslationtailminuslaplaceexponent}]{\textbf{#1.} }{\ \rule{0.5em}{0.5em}}
\newenvironment{proofoflemanegativeWHfactor}[1][Proof of Lemma \ref{lemmanegativeWHfactor}]{\textbf{#1.} }{\ \rule{0.5em}{0.5em}}
\newenvironment{proof}[1][Proof]{\textbf{#1.} }{\ \rule{0.5em}{0.5em}}

\def\CC{\mathbb{C}}
\newcommand{\rr}{\mathbb{R}}
\newcommand{\nn}{\mathbb{N}}
\newcommand{\B}{\mathcal{B}}
\newcommand{\p}{\mathbb{P}}
\newcommand{\E}{\mathbb{E}}

\allowdisplaybreaks
\section{Introduction}
\selectlanguage{english}
The Wiener-Hopf factorization for L\'evy processes has become a very important tool
due to its applications in several branches of applied probability, such
as insurance mathematics, theory of branching processes, mathematical finance and optimal control.
For instance, when the market is modelled by a L\'evy process, the positive Wiener-Hopf factor
allows to solve the optimal stopping problem corresponding to the pricing of a
perpetual call option,
while  the negative Wiener-Hopf factor is used to solve the
optimal stopping problem corresponding to the pricing of a
perpetual put option.
This negative Wiener-Hopf factor also arises in insurance mathematics in connection with scale functions appearing in   fluctuation identities. Such identities allow to obtain the joint distribution of the first passage time below a certain level and the position of the process at this time,
which is the classical ruin problem.

For a one-dimensional  L\'evy process $\x=\{\x(t),t\geq0\}$ we denote $S_t =\sup_{0\leq s\leq t}\x(s)$ and   $I_t =\inf_{0\leq s\leq t}\x(s)$.  The explicit distribution of  $S_t$ and $I_t$ in general is  difficult to obtain but the following relation holds. Let $e_q$ be an independent exponential random variable with parameter $q>0$.  The positive and negative Wiener-Hopf factors
of $\x$ are defined respectively as the random variables $S_{e_q} $ and $I_{e_q} $and they   satisfy the identity
\begin{equation}\label{igualdadWHfactors}
 \E\ci e^{ir\ S_{e_q} }\cd\E\ci e^{ir\ I_{e_q} }\cd=\frac{q}{q+\psi_\x(r)},\quad
 r\in\rr,
\end{equation}
where $\psi_\x(r)=-\log\E\ci e^{ir \x(1)}\cd$ is the characteristic exponent of $\x$. Only a few results are known for the explicit distribution of both Wiener-Hopf factors for processes with positive and negative jumps, see e.g. \cite{feller}, \cite{borovkov}, \cite{asmussenetal} and  \cite{kuznetsov}, \cite{kuznetsov2}.
While the distribution of
the positive Wiener-Hopf factor has been studied recently by several authors under some rather general conditions on the positive jumps (see,
e.g.  \cite{kuznetsov},
\cite{kuznetsovpeng}, \cite{lewismordecki} and the references therein),
the   distribution of the negative factor in these cases  has not be obtained explicitly.

In this paper we consider
 L\'evy processes $\x$  with  two-sided jumps  such that the positive jumps have rational Laplace transform,
 and with general negative jumps.  This class of L\'evy processes has been studied recently  in \cite{lewismordecki},
 where the authors obtained the explicit distribution of the positive Wiener-Hopf factor as well as
 asymptotic results for the tail of $S _\infty$. The particular case of L\'evy processes with  positive jumps which have phase-type distribution has been studied by \cite{asmussenetal} where the authors obtained the distributions of both Wiener-Hopf factors. The class of distributions having rational Laplace transforms is rich enough since
 it it dense in the class of  nonnegative distributions.  By inverting the Laplace transform of the random variable $-I_{e_q} $
 we provide  an explicit expression for the probability
density of the negative Wiener-Hopf factor $I_{e_q} $ in terms of given functions.
Under additional regularity assumptions on the L\'evy measure of $\x$  we obtain
asymptotic results as $u\to-\infty$ for the distribution function $F(u)$
of the  negative Wiener-Hopf factor
$I_{e_q} $. Our formula  for  the density of the negative Wiener-Hopf factor    generalizes the corresponding result in  \cite{asmussenetal}.

The paper is organized as follows: in Section \ref{basicnotions}
 we introduce basic concepts and notations and give some  preliminary results.
 In Section \ref{WHfactorssection} we obtain an expression for the Laplace transform
 of $-I_{e_q} $, which we invert in order to get
 an explicit formula for its probability density. In Section 4 we derive asymptotic results for the distribution of the negative Wiener-Hopf factor, while some relevant examples are given in Section \ref{examples}. In the final section we give the proof of the  auxiliary Lemma 5.

\section{Preliminary results}\label{basicnotions}\setcounter{equation}{0}

We consider the class of two-sided jumps  L\'evy processes $\x=\{\x(t),t\geq0 \}$, where
\begin{equation}\label{Xalpha}
\mathcal{X}(t)=ct +\gamma \mathcal{B}(t) + \mathcal{Z}(t)-\mathcal{S}(t),\ t\geq0.
\end{equation}
In the above expression, $c\geq0$ is a drift term, $\mathcal{B}=\{\mathcal{B}(t),t\geq0\}$
is a standard Brownian motion with variance parameter 2,
 $\subo=\{\subo(t),t\geq0\}$ is a pure jump L\'evy process having only positive jumps and $\mathcal{Z}=\{\mathcal{Z}(t),t\geq0\}$
is a compound Poisson process
with L\'evy measure $\lambda_1f_1(x)\,dx$, where $\lambda_1>0$ is constant. The function $f_1$ is assumed to be a probability
density with Laplace transform of the form
\begin{equation}\label{laplacef1}
\widehat{f}_1(r)=\frac{Q(r)}{\prod\limits_{i=1}^N(\alpha_i+r)^{n_i}},
\end{equation}
where $N, n_i\in \nn$ with $n_1+n_2+\dots +n_N=m$, $0\menor \alpha_1\menor \alpha_2\menor \dots \alpha_N$ are real numbers
and $Q(r)$ is a polynomial  of degree at most $m-1$.
Let $\Psi_\mathcal{S}(r)=-\log \E\ci e^{-r \subo(1)}\cd$ be
the Laplace exponent of the process $\subo$. It is known (see \cite{sato})
that $\Psi_\mathcal{S}(r) =\int_{0+}^\infty\pai1-e^{-rx}-rx h(x)\pad\nu_\mathcal{S}(dx)$,
where $h$ is a truncation function and $\nu_\subo$ is the L\'evy measure of $\subo$,
which satisfies $\int_{0+}^\infty \pai x ^2 \wedge 1\pad\nu_\subo(dx)<\infty$.
We also set $\overline{\mathcal{V}}_\subo(u)=\int_u^\infty\nu_\subo(dx)$.

For $\x$ given in (\ref{Xalpha}) we
consider the function
\begin{equation*}
  \Psi_\x(r)=cr+\gamma^2r^2+\lambda_1\pai\frac{Q(-r)}{\prod\limits_{j=1}^N(\alpha_j-r)^{n_j}}-1\pad-\Psi_\mathcal{S}(r).
\end{equation*}
Note that, for  $0\leq r\menor \alpha_1$, $\Psi_\x(r)=\log \E\ci e^{r\x(1)}\cd$.

When $\subo$ is a subordinator we replace $\Psi_\subo(r)$ in the above expression  by
$G_\subo(r)=\int^\infty_{0+}\pai1-e^{-rx}\pad\nu_\subo(dx),$
 and assume that the drift term $c$ includes the constant $\int_{0+}^\infty x h(x)\nu_\subo(dx)$.

In what follows we consider the sets $\CC_+=\{z\in\CC: Re(z)\geq0\}$ and $\CC_{++}=\{z\in\CC:Re(z)\mayor0\}$.
 We  consider the following cases:
\begin{enumerate}[\text{Case }A.]
 \item $c=\gamma=0$ and $\subo$ is a driftless subordinator other than a compound Poisson process or $\subo$ is a compound Poisson process such that $\E\ci\x(1)\cd>0$,
 \item $c\mayor0$, $\gamma=0$ and $\subo$ is a driftless subordinator,
 \item Any other case, except when $c=\gamma=0$ and $\subo$ is a compound Poisson processes with $\E\ci\x(1)\cd\leq 0$. In this case we also assume that $\int_{0+}^\infty(x^2\wedge x)\nu_\subo(dx)<\infty$.
\end{enumerate}

\begin{remark}
Assumption $\int_{0+}^\infty(x^2\wedge x)\nu_\subo(dx)<\infty$ is true, for instance, when $\int_{0+}^\infty x\nu_\subo(dx)<\infty.$
\end{remark}

 The following result from \cite{lewismordecki} holds for  the roots of the equation  $ \Psi_\x(r)-q=0,$ which we call generalized Cram\'er-Lundberg equation:
  \begin{lemma}\label{lemmarootsofGLEgeneralcase}
    Let $q\geq0$ and assume $\E\ci\x(1)\cd>0$ when $q=0$. Then:
    \begin{enumerate}[a)]
      \item In case A, the equation $ \Psi_\x(r)-q=0$ has $m$ roots in $\CC_{++}$,

    \item In cases B and C,  the equation $ \Psi_\x(r)-q=0$ has $m+1$ roots in $\CC_{++}$.
    \end{enumerate}
    In all the cases above, there is exactly one real root $\beta_1(q)$ in the interval $(0,\alpha_1)$, and
    it satisfies  $\lim\limits_{q\downarrow0}\beta_1(q)=0$.
    When $q=0$, $\beta_1(0)=0$ is a simple root of $ \Psi_\x(r)=0$ in all cases A, B and C.
    \end{lemma}

Let us assume that the equation $ \Psi_\x(r)-q=0$ has $R$ different roots in $\CC_{++}$,
denoted  by $\beta_1(q),\dots,\beta_R(q)$, with respectively multiplicities $k_1,k_2,\dots,k_R$,
where $\sum\limits_{j=1}^{R}k_j=m+1- 1_{\{case\ A\}}$, and $1_{\{case\ A\}}=1$ in case A and $1_{\{case\ A\}}=0$ in the other cases. We let $\beta_1(q)$ be the real root such
that $\beta_1(q)\in[0,\alpha_1)$, hence $k_1=1$.

The case when $q=0$ is taken in the limiting sense. 

When $\E\ci\x(1)\cd\leq 0$, we have $\p\ci I_\infty=-\infty\cd=1$, hence we have the following condition:

\begin{cond}\label{condiq0}
For $q=0$, we assume that $\E\ci\x(1)\cd>0$.
\end{cond}

For $a=0,1,\dots,m+1$, we define the linear operator $\mathcal{T}_{s,a}$ by the
expression
\begin{equation*}
 \mathcal{T}_{s,a}f(u)=\int_u^\infty(y-u)^ae^{-s(y-u)}f(y)dy,
\end{equation*}
for all measurable, nonnegative functions $f$ and complex numbers $s$ such that the integral above exists and is finite.
If $\nu$ is a measure such that
$\int_u^\infty(y-u)^ae^{-s(y-u)}\nu(dy)$ exists, we define for $a=0,1,\dots,m+1$,
\begin{equation}\label{operadorT}
 \mathcal{T}_{s,a}\nu(u)=\int_u^\infty(y-u)^ae^{-s(y-u)}\nu(dy),
\end{equation}
and denote the Laplace transforms of these two operators by $\widehat{\mathcal{T}}_{s,a} f$
and $\widehat{\mathcal{T}}_{s,a} \nu$.
When $a=0$, we obtain the Dickson-Hipp operator $T_sf$ defined
in \cite{dicksonhipp}
 and write $\mathcal{T}_{s}f(u)=\int_u^\infty e^{-s(y-u)}f(y)dy$, with the corresponding modification
when $f$ is replaced by a measure $\nu$.
We shall use
the following elementary properties and lemma:
\small
\begin{equation}\label{laplaceoperatorT}
 \widehat{T}_sf(r)=\frac{\widehat{f}(r)-\widehat{f}(s)}{s-r},\quad \widehat{T}_s\nu(r)=\frac{\int_{0+}^\infty \pai e^{-rx}-e^{-sx}\pad\nu(dx)}{s-r}.
\end{equation}
\normalsize

\begin{lemma}\label{lemaderivadaseintegrales}
 Let $f$ be a function (or a measure) such that $\mathcal{T}_{s,k} f(u)$ exists for every $s\in\CC_{++}$, $k\in \nn\cup\{0\}$
 and $u\mayor0$. For each $r\in \CC_+$, $s\in\CC_{++}$ and $k\in\nn\cup\{0\}$ there holds $\frac{\partial^k}{\partial s^k}\widehat{T}_sf(r)=(-1)^k\widehat{\mathcal{T}}_{s,k}f(r)$.
\end{lemma}

The following result follows from Theorem 6.16 in \cite{kyprianou}.
\begin{lemma}\label{lemmakyprianou}
Let $S_{e_q}$ be the positive Wiener-Hopf factor of a L\'evy process, other than a compound Poisson process, and denote by $\kappa$ the joint Laplace exponent for the bivariate subordinator representing the ascending  ladder process $( {\mathbb{L}}^{-1},\mathbb{H}), $
and by   $\Lambda (dx,dy)$ its  bivariate L\'evy measure.
   Then  there exist $b\geq 0$ such that, for $r\geq0$ and $q\mayor0$, it holds
\begin{equation*}
b  r+\int _{0+}^\infty\pai1-e^{-ry}\pad \int_{0+}^\infty e^{-q x}\Lambda (dx,dy)=\kappa(q,r)-\kappa(q,0),
\end{equation*}
and
\begin{equation}\label{laplacepositiveWHfactor}
 \E\ci e^{-r  S_{e_q}}\cd=
E\ci e^{-r \mathbb{H}_q \pai e_{\kappa(q,0)}\pad}\cd=\frac{\kappa(q,0)}{ \kappa(q,r) }.
\end{equation}
Here $e_{\kappa(q,0)}$ is an exponential random variable
with mean $1/\kappa(q,0),$ independent of the L\'evy process.
It also holds 
\begin{equation}\label{kapa}
q-\Psi_{\x}(r)=\kappa(q, -ir)\widehat {\kappa}(q, ir),
\end{equation}
where $\widehat{\kappa}$ is the joint Laplace exponent for the bivariate subordinator representing the descending  ladder process $( {\mathbb{L}}^{-1},\mathbb{\widehat{H}}). $
\end{lemma}

In order to simplify our notations, we define the following constants:
$$E(j,a,q)=\binom{k_j-1}{a}\frac{(-1)^{1-k_j+a}}{(k_j-1)!}\frac{\partial^{k_j-1-a}}{\partial s^{k_j-1-a}}\ci\frac{\prod\limits_{l=1}^N(\alpha_l-s)^{n_l}(\beta_j(q)-s)^{k_j}}{\prod\limits_{l=1}^R(\beta_l(q)-s)^{k_l}}\cd_{s=\beta_j(q)},$$
$$E_*(j,a,q)=\binom{k_j-1}{a}\frac{(-1)^{1-k_j+a}}{(k_j-1)!}\frac{\partial^{k_j-1-a}}{\partial s^{k_j-1-a}}\ci\frac{\prod\limits_{l=1}^N(\alpha_l-s)^{n_l}(\beta_j(q)-s)^{k_j}}{\prod\limits_{l=1}^R(\beta_l(q)-s)^{k_l}}s\cd_{s=\beta_j(q)},$$
for each $j=1,2,\dots,R$. The constants $E(j,0,q)$ and $E(j,a,q)$ for $a>0$ correspond, respectively, to those given in expressions (2.4)
 and (2.5) in \cite{lewismordecki}.

We define the functions
\begin{align}
\ell_q(u)= \sum\limits_{j=1}^R\sum\limits_{a=0}^{k_j-1}E(j,a,q)\mathcal{T}_{\beta_j(q),a}\nu_\subo(u),  
\quad \EL_q(u)=        \sum\limits_{j=1}^R\sum\limits_{a=0}^{k_j-1}E_*(j,a,q)\mathcal{T}_{\beta_j(q),a}\overline{\mathcal{V}}_\subo(u),\quad q\geq0,\nonumber %
\end{align}
and the measure
\begin{equation}\label{funcionesL}
\chi_{q,\subo}(dx )=\left\{
 \begin{array}{ll}
  \nu_\subo(dx)+\ell_q(x)dx&\mbox{in case A },\\
  &\\
    \ell_q(x)dx&\mbox{in case B},\\
  &\\
  \ci\overline{\mathcal{V}}_\subo(x)-\EL_q(x)\cd dx&\mbox{in case C}.
 \end{array}
 \right.
\end{equation}

\section{Main results}\setcounter{equation}{0}\label{WHfactorssection}

In this section we obtain an explicit expression for the probability density of
the negative
Wiener-Hopf factor $I _{e_q}$of
the process $\x$ defined in (\ref{Xalpha}). The results presented for $q=0$ are all under the assumption that Condition \ref{condiq0} holds.

For $a>0$  let $\e_a(x)$ denote the exponential density $\e_a(x)=a e^{-ax}$, $ x>0$ and define, for $q\ge0$, the function $\widehat{W}_q$ by
\begin{equation}\label{laplaceWcasogeneral}
 \widehat{ W}_q(r)=\frac{\kappa(q,-r)}{\ci q- \Psi_\x(r)\cd}, \quad r\ge0,
\end{equation}
where $\kappa$ is given in Lemma \ref{lemmakyprianou}. By Theorem 2.2 in \cite{lewismordecki}, we know that

\begin{equation}\label{kappa1}
\kappa(q,r)=\frac{\prod_{j=1}^R(\beta_j(q)+r)^{k_j}}{\prod_{l=1}^N(\alpha_l+r)^{n_l}}
\end{equation}

 so $\kappa(q,0)=\frac{\prod_{j=1}^R\beta_j^{k_j}(q)}{\prod_{l=1}^N\alpha_l^{n_l}(q)}$. Since $\beta_1(0)=0$, it follows that $\kappa(0,0)=0$. Hence, using that from L'H\^opital's rule  $\lim_{\beta_1(q) \to 0}{ \frac{\Psi_{\x}(\beta_{1}(q))}{\beta_{1}(q))}}=\E(\x(1)),$
 we obtain

$$\lim_{q\downarrow0}\frac{q}{\kappa(q,0)}=\lim_{q\downarrow0}\frac{q}{\beta_1(q)}\frac{\prod_{l=1}^N\alpha_l^{n_l}}{\prod_{j=2}^R\beta_j^{k_j}(q)}=\E\ci\x(1)\cd\frac{\prod_{l=1}^N\alpha_l^{n_l}}{\prod_{j=2}^R\beta_j^{k_j}(0)}.$$

We define for $q>0$, $a(q)=q/\kappa(q,0)$ and $a(0)=\lim_{q\downarrow0}a(q)$. 

 Hence, we have the following result.

\begin{lemma}
The Laplace transforms of $-I _{e_q}$ for $q\mayor0$ and
$-I_\infty $ satisfy   the following equalities for $r\geq0$:
\begin{equation}\label{mainteo1parteA}
 \E\ci e^{-r \ci-I _{e_q}\cd}\cd=a(q) \widehat{ W}_q(r)\quad\text{ and }\quad\E\ci e^{-r \ci-I _\infty\cd}\cd=a(0)\widehat{W}_{0}(r)
\end{equation}

Hence, for $q\geq0$,  $a(q) W_q$ is the density of $-I_{e_q} $ .
\end{lemma}

\begin{proof}
Using (\ref{igualdadWHfactors}), Lemma \ref{lemmakyprianou} and the relation $\psi_\x(s)=-\Psi_\x(is)$, we get
\begin{equation}\label{relacionhatkappaykappa}
\E\ci e^{is I _{e_q}}\cd=\frac{q}{q-\Psi_\x(is)}\pai\frac{\kappa(q,-is)}{\kappa(q,0)}\pad.
\end{equation}
The function on the right-hand side can be analytically extended to negative part of the imaginary axis. Hence the result follows taking $s=-ir$ for $r\geq0$.

The case for $q=0$ follows by taking limits when $q\downarrow0$. 
\end{proof}

In the following result we invert $a(q)\widehat{W}_q$.

\begin{teo}\label{laplacefactorWienerHopfnegativo}\leavevmode
\begin{enumerate}[a)]

\item The function $\widehat{W}_q$ satisfies the equalities
\begin{equation}\label{laplaceWienerHopf}
a(q) \widehat{ W}_q(r)=\left\{
\begin{array}{ll}
\frac{\widehat{\e}_{a(q)}(r)}{1-\pai1-\widehat{\e}_{a(q)}(r)\pad\pai1-\widehat{\colachi}_{q,\subo}(r)\pad}, &\text{ in cases A and C with }\gamma=0,\\
\frac{\frac{a(q)}{c}}{1-\frac{1}{c}\widehat{\chi}_{q,\subo}(r)}, &\text{ in case B},\\
\frac{\widehat{\e}_{a(q)\gamma^{-2}}(r)}{1+\gamma^{-2}\pai1-\widehat{\e}_{a(q)\gamma^{-2}}(r)\pad\widehat{\colachi}_{q,\subo}(r)}, &\text{ in case C with } \gamma>0.
\end{array}\right.
\end{equation}
\item For $q\geq0$ and $u\ge0$ the negative Wiener-Hopf factor $I_{e_q}$ has a generalized density function given by:
 \begin{align}
 &a(q) W_q(u)\nonumber \\
 &=\left\{
 \begin{array}{ll}
\e_{a(q)}*\sum\limits_{n=0}^\infty\sum\limits_{k=0}^n\binom{n}{k}(-1)^k\pai\colachi_{q,\subo}+\e_{a(q)}-\e_{a(q)}*\colachi_{q,\subo}\pad^{*k}(u)&\text{ in cases A and C with }\gamma=0,\\
&\\
\frac{a(q)}{c}\delta_0(u)+\frac{a(q)}{c}\sum\limits_{n=1}^\infty\pai\dfrac{1}{c}\pad^n\chi_{q,\subo}^{*n}(u),&\text{ in case B},\\
 &\\
 \e_{a(q)\gamma^{-2}}*\sum\limits_{n=0}^\infty\pai-\frac{1}{\gamma^2}\pad^n\Bigg(\colachi_{q,\subo}-\e_{a(q)\gamma^{-2}}*\colachi_{q,\subo}\Bigg)^{*n}(u)&\text{ in case C with }\gamma>0,\nonumber 
 \end{array}\right.
 \end{align}
where $\delta_0$ is Dirac's delta function.
\end{enumerate}
\end{teo}

In order to  prove Theorem \ref{laplacefactorWienerHopfnegativo} we need the following lemma. Its proof is technical and lengthly and   is deferred to section 6.
\begin{lemma}\label{lemmanegativeWHfactor}
For $q\geq0$ we have:
\begin{equation}\label{lemanegativeWHfactor}
\ci q- \Psi_\x(r)\cd\pai \kappa(q,-r)\pad^{-1}=\left\{
\begin{array}{ll}
 a(q)+G_\mathcal{S}(r)+\widehat{\ell}_q(0)-\widehat{\ell}_q(r), &\text{ in Case A},\\
 a(q)+\widehat{\ell}_q(0)-\widehat{\ell}_q(r), &\text{ in Case B},\\
 a(q) +\gamma^2 r-\frac{\Psi_\subo(r)}{r}-\big[\widehat{\EL}_q(0)-\widehat{\EL}_q(r)\big], &\text{ in Case C}.
 \end{array}\right.
\end{equation}
\end{lemma}

\begin{proofofmainteo1}
Clearly, part b) follows inverting (\ref{laplaceWienerHopf}). To prove  (\ref{laplaceWienerHopf}) we assume $q>0$. The case $q=0$ follows by letting $q\downarrow0$.

From (\ref{mainteo1parteA}),
 (\ref{funcionesL}),(\ref{laplaceWcasogeneral}), (\ref{lemanegativeWHfactor}) and the definition of $a(q)$ we obtain
 \begin{equation}\label{caso A}  
  \E\ci e^{-r \ci-I _{e_q}\cd}\cd=a(q) \widehat{ W}_q(r)
  =\frac{a(q)}{a(q)+\int\limits_{0+}^\infty(1-e^{-rx})\chi_{q,\subo}(dx )} =\frac{q}{q+\int\limits_{0+}^\infty(1-e^{-rx})\kappa(q,0)\chi_{q,\subo}(dx )}.
\end{equation}
 
To obtain (\ref{laplaceWienerHopf}) in case A, we use
\begin{equation}\label{laplaceE}
\widehat{\e}_{a(q)}(r)=\frac{a(q)}{a(q)+r}.
\end{equation} and apply Fubini's Theorem to $\int\limits_{0+}^\infty(1-e^{-rx})\chi_{q,\subo}(dx )$. This yields,
\begin{align}
  a(q) \widehat{ W}_q(r)&=\frac{a(q)}{a(q)+\int\limits_{0+}^\infty(1-e^{-rx})\chi_{q,\subo}(dx )}
  =\frac{a(q)}{a(q)+r\widehat{\colachi}_{q,\subo}(r)}
    =\frac{\widehat{\e}_{a(q)}(r)}{1-\pai1-\widehat{\e}_{a(q)}(r)\pad\pai1-\widehat{\colachi}_{q,\subo}(r)\pad},\nonumber 
  \end{align}
Hence (\ref{laplaceWienerHopf}) follows in this case. 
In case B, (\ref{laplaceWcasogeneral}) and (\ref{lemanegativeWHfactor}) we have
\begin{align}
  a(q)\widehat{W}_q(r)&=\frac{a(q)}{a(q)+\widehat{\ell}_q(0)-\widehat{\ell}_q(r)}  =\frac{\frac{a(q)}{a(q)+\widehat{\ell}_q(0)}}{1-\frac{1}{a(q)+\widehat{\ell}_q(0)}\widehat{\ell}_q(r)}.\label{igualdad2casoB}
\end{align}
Due to (\ref{igualdadadelta})  we have $a(q)+\widehat{\ell}_q(0)=c$, and
from  (\ref{funcionesL}) it follows that $\widehat{\ell}_q(r)=\widehat{\chi}_{q,\subo}(r)$.
Substituting these two equalities into (\ref{igualdad2casoB}) and using (\ref{mainteo1parteA}) gives (\ref{laplaceWienerHopf}).

We now deal with case C. Using (\ref{laplaceWcasogeneral}) and (\ref{lemanegativeWHfactor}), we obtain $$ a(q)\widehat{W}_q(r)=\frac{a(q)}{a(q)+\gamma^2 r-\frac{\Psi_\subo(r)}{r}-\Big[\widehat{\EL}_q(0)-\widehat{\EL}_q(r)\Big]}.
$$

Now we apply Fubini's theorem to $\Psi_\subo(r)/r$ to obtain, for $\gamma>0$:
\begin{align}
 a(q)\widehat{W}_q(r)
 &=\frac{\frac{a(q)\gamma^{-2}}{a(q)\gamma^{-2}+r}}{1+\gamma^{-2}\frac{r}{a(q)\gamma^{-2}+r}\widehat{\colachi}_{q,\subo}(r)}
 =\frac{\widehat{\e}_{a(q)\gamma^{-2}(r)}}{1+\gamma^{-2}\pai1-\widehat{\e}_{a(q)\gamma^{-2}(r)}\pad\widehat{\colachi}_{q,\subo}(r)},\nonumber
\end{align}
where we have used (\ref{laplaceE}) with $a(q)$ replaced by $a(q)\gamma^{-2}$. When $\gamma=0$ it holds,
\begin{align}
 a(q)\widehat{W}_q(r)
  &=\frac{\frac{a(q)}{a(q)+r}}{1+\frac{r}{a(q)+r}\widehat{\colachi}_{q,\subo}(r)}
 =\frac{\widehat{\e}_{a(q)(r)}}{1+\pai1-\widehat{\e}_{a(q)(r)}\pad\widehat{\colachi}_{q,\subo}(r)},\nonumber
\end{align}
and we obtain (\ref{laplaceWienerHopf}) using (\ref{mainteo1parteA}).
\end{proofofmainteo1}

\begin{lemma} For all $q\ge0$ the measure $\kappa (q,0) \chi_{q,\subo}$ is the L\'evy measure  of  $-I_{e_q}.$  
  
 \end{lemma}
  \begin{proof} 
The non-negative random variable $-I_{e_q} $ is  infinitely divisible,
with   Laplace transform 
 \begin{equation}\label{laplacefactornegativo2} 
 \E\ci e^{-r\pai-I_{e_q} \pad}\cd=\exp\left\{ -\int_{0+}^\infty \pai1-e^{-rx}\pad\nu_q(dx)\right\},
 \end{equation} where the measure
 \begin {equation} \label{medidanu}
 \nu_q(dx) =\int_{0}^\infty t^{-1}e^{-q t}\p\ci -I_t \in dx\cd dt
 \end{equation}
 is the L\'evy measure of  $-I_{e_q}$ (see e.g. Lema 6.17 in \cite{kyprianou}).  
 On the other hand, from the formula for Frullani's integral, we have for $\alpha,\beta>0$ and $z\leq 0$,
\begin{equation}\label{frullanisintegral}
 \pai\frac{\alpha}{\alpha-z}\pad^\beta=\exp\left\{-\int_0^\infty(1-e^{zt})\beta t^{-1}e^{-\alpha t}dt\right\}.
\end{equation}
 
Let $\mathcal{N}_q$ denote the subordinator with L\'evy measure $\kappa(q,0)\chi_{q,\subo},$ and denote its Laplace exponent by  $\Psi_q$.  From (\ref{caso A}) we have
$\E\ci e^{-r(-I_{e_q} )}\cd=
 \frac{q}{q+\Psi_q(r)},$  hence using (\ref{frullanisintegral}) with $\alpha=q$, $\beta=1$ and $z=-\Psi_q(r)$
we obtain
\begin{equation}\label{frullaniA}
 \E\ci e^{-r\pai-I_{e_q} \pad}\cd=
\exp\left\{ -\int_0^\infty\pai1-e^{-t\Psi_q(r)}\pad t^{-1}e^{-q t}dt\right\}.
\end{equation}
Since $1-e^{-t\Psi_q(r)}=\int_0^\infty \pai1-e^{-rx}\pad\p\ci \mathcal{N}_q(t)\in dx\cd$, setting
\begin{equation}\label{medida}
\pi(dx)=\int_{0}^\infty t^{-1}e^{-q t}\p\ci \mathcal{N}_q(t) \in dx\cd dt,
\end{equation} 
and using
Fubini's theorem in (\ref{frullaniA}) it follows that
\begin{equation}\label{frullani2}
 \E\ci e^{-r\pai-I_{e_q} \pad}\cd=
\exp\left\{ -\int_0^\infty\pai1-e^{-rx}\pad \pi(dx)\right\}. 
\end{equation} 
Now from (\ref{laplacefactornegativo2})  we deduce that 
\begin{equation}\label{nu}
\pi= \nu_q.
\end{equation}
Using (\ref{medida}) and (\ref{medidanu}) we obtain the result.
\end{proof}

\begin{remark} Since $-I_{e_q}=\widehat{H}(e_{\widehat{\kappa}(q,0)})$ in distribution  \cite[Theorem 6.16]{kyprianou},  it follows that  the measure $\kappa (q,0) \chi_{q,\subo}$ is also the L\'evy measure of the descending ladder-height process $\widehat{H}$ corresponding to $\x(t),$ killed at the uniform rate $\widehat{\kappa}(q,0).$
\end{remark}

\section{Asymptotic behavior of  the negative Wiener-Hopf factor}\setcounter{equation}{0}

 For $u>0$ we denote
$F_{I_{e_q} }(-u)=\p\ci I_{e_q} <-u\cd$

We obtain asymptotic expressions for $F_{I_{e_q} }(-u)$ when $u\to\infty$. For this, we use the following technical result.

\begin{lemma}
The equality
\begin{align}
r\widehat{F}_{I_{e_q}}(r)
&=\frac{\sum\limits_{j=1}^{m+2}A_j'r^j+\Psi_\subo(r)\prod\limits_{l=1}^N\alpha_l^{n_l}+\sum\limits_{j=1}^mA_j r^j\Psi_\subo(r)-\widehat{\kappa}(q,0)\sum\limits_{j=1}^{m+1}B_jr^j}{q\prod\limits_{l=1}^N\alpha_l^{n_l}+\sum\limits_{j=1}^{m+2}A_j'r^j+\Psi_\subo(r)\prod\limits_{l=1}^N\alpha_l^{n_l}+\sum\limits_{j=1}^mA_j r^j\Psi_\subo(r)}\label{expansionPsiX}
\end{align}
holds
for some constants $A_j,j=1,\dots,m$, $B_k,k=1,\dots,m+1$ and $A_l',l=1,\dots,m+2$.
\end{lemma}

\begin{proof}
Using that $r\widehat{F}_{I_{e_q}}(r)=1-\widehat{f}_{-I_{e_q}}(r)$, where $\widehat{f}_{-I_{e_q}}(r)$ denotes the Laplace transform of the density of $-I_{e_q}$, we obtain
as in (\ref{laplacepositiveWHfactor}) \begin{equation}
r\widehat{F}_{I_{e_q}}(r)=\frac{\widehat{\kappa}(q,r)-\widehat{\kappa}(q,0)}{\widehat{\kappa}(q,r)}.\label{kappas}
\end{equation}
On the other hand  from (\ref{kapa}) and (\ref{kappa1}) it follows \begin{equation}
\widehat{\kappa}(q,r)=\prod_{l=1}^N\pai\alpha_l-r\pad^{n_l}\frac{q-\Psi_\x(r)}{\prod_{j=1}^R\pai\beta_j(q)-r\pad^{k_j}}\label{kappagorro}.
\end{equation}
Since 
$$q-\Psi_\x(r)=q-cr-\gamma^2r^2-\lambda_1\pai\frac{Q(-r)}{\prod\limits_{j=1}^N(\alpha_j-r)^{n_j}}-1\pad+\Psi_\mathcal{S}(r),$$
and $\lambda_1\pai\frac{Q(-r)}{\prod\limits_{j=1}^N(\alpha_j-r)^{n_j}}-1\pad$ can be written as the quotient of some polynomial $\mathcal{P}_1$ with degree $m$ and $\prod\limits_{j=1}^N(\alpha_j-r)^{n_j}$, which is also a polynomial of degree $m$, we have
\begin{equation}\label{laplaceexponentwithpolynomials}
q-\Psi_\x(r)=q-cr-\gamma^2r^2-\frac{\mathcal{P}_1(r)}{\prod\limits_{j=1}^N(\alpha_j-r)^{n_j}}+\Psi_\mathcal{S}(r).
\end{equation}
Since $\frac{Q(0)}{\prod\limits_{j=1}^N\alpha_j^{n_j}}=1$, we obtain that $\mathcal{P}_1$ is a polynomial with constant term  $0$. Using  $\prod\limits_{j=1}^N(\alpha_j-r)^{n_j}=\sum_{j=0}^mA_j r^j$, it follows that
\begin{align}
\prod\limits_{j=1}^N(\alpha_j-r)^{n_j}(q-\Psi_\x(r))&=qA_0+q\sum_{j=1}^mA_j r^j-c\sum_{j=0}^mA_j r^{j+1}-\gamma^2\sum_{j=0}^mA_j r^{j+2}\nonumber\\
&-\mathcal{P}_1(r)+A_0\Psi_\subo(r)+\sum_{j=1}^mA_j r^j\Psi_\subo(r)\nonumber\\
&=qA_0+\sum\limits_{j=1}^{m+2}A_j'r^j+A_0\Psi_\subo(r)+\sum_{j=1}^mA_j r^j\Psi_\subo(r)\label{expansionpolinomialdePsiX}.
\end{align}
Evaluating $\prod\limits_{j=1}^N(\alpha_j-r)^{n_j}$ at $r=0$ we obtain $A_0=\prod_{l=1}^N\alpha_l^{n_l}$. Moreover, setting $r=0$ in (\ref{kappagorro}) gives
\begin{equation}\label{qA0}
q\prod\limits_{l=1}^N\alpha_l^{n_l}=\widehat{\kappa}(q,0)\prod\limits_{j=1}^R\beta_j^{k_j}(q).
\end{equation}
Hence, using that $\prod_{j=1}^R\pai\beta_j(q)-r\pad^{k_j}$ can be expressed as $\sum_{j=0}^{m+1}B_jr^j$ and substituting (\ref{qA0}) and (\ref{expansionpolinomialdePsiX}) into (\ref{kappas}), it follows that
\begin{align}
r\widehat{F}_{I_{e_q}}(r)
&=\frac{\sum\limits_{j=1}^{m+2}A_j'r^j+\Psi_\subo(r)\prod\limits_{l=1}^N\alpha_l^{n_l}+\sum\limits_{j=1}^mA_j r^j\Psi_\subo(r)-\widehat{\kappa}(q,0)\sum\limits_{j=1}^{m+1}B_jr^j}{q\prod\limits_{l=1}^N\alpha_l^{n_l}+\sum\limits_{j=1}^{m+2}A_j'r^j+\Psi_\subo(r)\prod\limits_{l=1}^N\alpha_l^{n_l}+\sum\limits_{j=1}^mA_j r^j\Psi_\subo(r)}.\nonumber
\end{align}
\end{proof}

In what follows we write $f\approx c g$
for any two nonnegative functions $f$ and $g$ on $[0,\infty)$  such that $\lim\limits_{u\to\infty}\frac{f(u)}{g(u)}=c$, with $c\neq 0$. 
Now we can derive the first asymptotic expression for $F_{I_{e_q} }$.

\begin{propo}\label{proporeferi}
If  $r^{-\xi}\Psi_\subo(r)\to D$ as $r\downarrow 0$  for some $\xi\in(0,1)$ and  some positive constant $D$, then 
$$F_{I_{e_q} }(-u)\approx \frac{D}{q\Gamma(1-\xi)}u^{-\xi},\quad u\to\infty.$$
\end{propo}
 
\begin{proof}
Due to Theorem 4 in \cite{feller}, page 446, we only need to prove that 

\begin{equation}\label{qepd}
r^{-\xi+1}\widehat{F}_{I_{e_q} }(r)\to Dq^{-1}.
\end{equation}
From (\ref{expansionPsiX}) we have
\begin{align}
r^{-\xi+1}\widehat{F}_{I_{e_q} }(r)=\frac{\sum\limits_{j=1}^{m+2}A_j'r^{j-\xi}+r^{-\xi}\Psi_\subo(r)\prod\limits_{l=1}^N\xi_l^{n_l}+\sum\limits_{j=1}^mA_j r^{j-\xi}\Psi_\subo(r)-\widehat{\kappa}(q,0)\sum\limits_{j=1}^{m+1}B_jr^{j-\xi}}{q\prod\limits_{l=1}^N\xi_l^{n_l}+\sum\limits_{j=1}^{m+2}A_j'r^j+\Psi_\subo(r)\prod\limits_{l=1}^N\xi_l^{n_l}+\sum\limits_{j=1}^mA_j r^j\Psi_\subo(r)}.
\end{align}
Since all the polynomial terms in the numerator have no constant term and $\xi\in(0,1)$, we have ${j-\xi}>0$, hence letting $r\downarrow0$ and using the hypothesis on $r^{-\xi}\Psi_\subo(r)$, we obtain (\ref{qepd}). 
\end{proof}  
 
Let us denote by $\colachi_{q,\subo}$ the tail of the L\'evy measure $\chi_{q,\subo}$.

\begin{propo}
In case B, if $1-\frac{\colachi_{q,\subo}(u)}{\widehat{\chi}_{q,\subo}(0)}$ is a subexponential distribution,
then $F_{I_{e_q} }(-u)\approx q^{-1}\kappa(q,0)\colachi_{q,\subo}(u)$.
\end{propo}
\begin{proof}
First we note that from (\ref{igualdadadelta}) we get $c=a(q)+\widehat{\ell}_q(0)$, hence $a(q)=c-\widehat{\chi}_{q,\subo}(0)$, and
\begin{align}
  \widehat{F}_{I_{e_q} }(r)
  &=\frac{1}{r}\pai\frac{\frac{1}{c}\widehat{\chi}_{q,\subo}(0)-\frac{1}{c}\widehat{\chi}_{q,\subo}(r)}{1-\frac{1}{c}\widehat{\chi}_{q,\subo}(r)}\pad
  =\frac{\frac{1}{c}\widehat{\colachi}_{q,\subo}(r)}{1-\frac{1}{c}\widehat{\chi}_{q,\subo}(r)},\nonumber
\end{align}
where in the last equality we used that for a probability density $f$ with tail $\overline{F}$, we have $\widehat{\overline{F}}(r)=r^{-1}\pai1-\widehat{f}(r)\pad$.
It follows that
\begin{equation}\label{colacasoB}
F_{I_{e_q} }(-u)=\frac{1}{c}\colachi_{q,\subo}*\sum\limits_{n=0}^\infty\pai\frac{1}{c}\pad^n\colachi_{q,\subo}^{*n}(u).
\end{equation}

Now we set $p=\frac{\widehat{\chi}_{q,\subo}(0)}{a(q)+\widehat{\chi}_{q,\subo}(0)}$. Then $p\in(0,1)$ and from (\ref{igualdadadelta}) it follows that
 $p=\frac{\widehat{\chi}_{q,\subo}(0)}{c}$. Using (\ref{colacasoB}) we get
 \begin{align}
  F_{I_{e_q} }(-u)&= \dfrac{1}{c}\colachi_{q,\subo}*\sum\limits_{n=0}^\infty\pai\dfrac{1}{c}\pad^n\chi_{q,\subo}^{*n}(u)
  = p\dfrac{\colachi_{q,\subo}}{\widehat{\chi}_{q,\subo}(0)}*\sum\limits_{n=0}^\infty p^n\pai\dfrac{\chi_{q,\subo}}{\widehat{\chi}_{q,\subo}(0)}\pad^{*n}(u).\nonumber
 \end{align}
 Let us define the probability distribution
 $H(u)=(1-p)\int\limits_0^u\sum\limits_{n=0}^\infty p^n\pai\dfrac{\chi_{q,\subo}}{\widehat{\chi}_{q,\subo}(0)}\pad^{*n}(x)dx$,
 and denote its density by $h$.
Since $\colachi_{q,\subo}(u)/\widehat{\chi}_{q,\subo}(0)$ is the tail of a proper distribution, due to Corollary 3 in \cite{embrechtsetal}
and the assumption that $1-\colachi_{q,\subo}/\widehat{\chi}_{q,\subo}(0)$ is subexponential, we obtain
\begin{equation}\label{asymptoticG}
 \overline{H}(u)\approx \frac{p}{1-p}\frac{\colachi_{q,\subo}(u)}{\widehat{\chi}_{q,\subo}(0)},
\end{equation}
hence $H$ is a subexponential distribution. From Lemma 2.5.2 in \cite{rolskietal} and (\ref{asymptoticG}), we have
\begin{align}
 \lim\limits_{u\to\infty}\frac{F_{I_{e_q} }(-u)}{\colachi_{q,\subo}(u)/\widehat{\chi}_{q,\subo}(0)}&=\lim\limits_{u\to\infty}\frac{\frac{p}{1-p}\frac{\colachi_{q,\subo}}{\widehat{\chi}_{q,\subo}(0)}*h(u)}{\colachi_{q,\subo}(u)/\widehat{\chi}_{q,\subo}(0)}
 =\frac{p}{1-p}=\widehat{\chi}_{q,\subo}(0)q^{-1}\kappa(q,0),\nonumber
\end{align}
which implies the result.\hfill
\end{proof}

Let us define $\Pi(x)=\frac{\int_1^x \chi_{q,\subo}(dy)}{\int_1^\infty  \chi_{q,\subo}(dy)}$ for $x>1$. We recall that a probability distribution $F$ is a
subexponential distribution if $\overline{F^{*2}}(x)\approx 2\overline{F}(x)$
as $x\to\infty$.
\begin{propo}
 Suppose that $\Pi$ is a subexponential distribution and  set $\overline{\nu_q}(u)=\int\limits_u^\infty\nu_q(dy)$ where $\nu_q$ is the L\'evy measure of $-I_{e_q}.$
 Then the random variable $-I_{e_q} $ has a subexponential distribution, and
 \begin{equation}\label{colaasintoticamedidapi}
  F_{I_{e_q} }(-u)\approx \overline{\nu_q}(u),\ u\to\infty.
 \end{equation}
\end{propo}

\begin{proof}
 The assertion that $-I_{e_q} $ has a subexponential distribution and that (\ref{colaasintoticamedidapi}) holds, follow from
(\ref{frullani2}) and  Theorem 1 in \cite{embrechtsetal}.
\end{proof}

\section{Examples}\label{examples}\setcounter{equation}{0}
In this section we apply the results from the previous section to several particular examples in which we obtain 
simple asymptotic expressions for the negative Wiener-Hopf factor $F_{I_{e_q} }$.
For simplicity, in cases A and C we will assume that the 
roots of $\Psi_\x(r)-q=0$ in $\CC_{++}$
are all different. This assumption holds e.g. when the density $f_1$ is a convex combination of exponential densities.

\textbf{Example 1 (case A)}.
 We take $\Psi_\subo(r)=r^\xi$, for $\xi \in(0,1)$, hence $\subo$ is an $\xi$-stable subordinator and  the assumptions on Proposition \ref{proporeferi} hold. Hence,
$$F_{I_{e_q} }(-u)\approx \frac{1}{\Gamma(1-\xi)}u^{-\xi} \mbox{ as }u\to\infty.$$

\textbf{Example 2 (Case B)}.
1. Let us
suppose that $G_\subo(r)=\lambda_2\widehat{f}_2(r)-\lambda_2$, i.e. $\subo$ is a compound Poisson
process with L\'evy measure $\lambda_2f_2(x)$ with $\lambda_2\mayor0$. In this case the
resulting L\'evy risk process is the classical two-sided jumps risk process.
Therefore
$$\chi_{q,\subo}(u)=\lambda_2\sum\limits_{j=1}^R\sum\limits_{a=0}^{k_j-1}E(j,a,q)
\int_u^\infty(y-u)^a e^{-\beta_j(q)(y-u)}f_2(y)dy.$$
When $k_j=1$ for all $j=1,2,\dots,R$, and
$f_2$ is a mixture of exponential densities, the above expression can be easily calculated.
Let us consider the particular case when $f_1(x)= qe^{-qx}$, $x>0$, $\lambda_2=1$ and $f_2(x)=pe^{-px}$, $x>0$.
In this case the generalized
Lundberg equation $\Psi_\x(r)-q=0$ has two real roots $\beta_1(q)$ and $\beta_2(q)$ such that
$0<\beta_1(q)<q<\beta_2(q)$. Hence
$$\chi_{q,\subo}(u)=\pai\frac{q-\beta_1(q)}{\beta_2(q)-\beta_1(q)}\frac{1}{\beta_1(q)+p}+\frac{\beta_2(q)-q}{\beta_2(q)-\beta_1(q)}\frac{1}{\beta_2(q)+p}\pad pe^{-pu}:=C_q pe^{-pu},$$
which means that the associated subordinator $\mathcal{N}_{2,q}$ is a compound Poisson process with intensity $C_q$ and jump sizes
with density $f_2$.
 In this case we obtain an explicit expression for $F_{I_{e_q}}$ and its Laplace transform:
\begin{equation}\label{CP}
\widehat{F}_{I_{e_q} }(r)=\frac{\frac{C_q p}{c}}{\frac{p(c-C_q)}{c}+r}\quad\mbox{and}\quad
 F_{I_{e_q} }(-x)=\frac{C_q p}{c}e^{-p(c-C_q) x}.
 \end{equation}

2. Now let us suppose that $\overline{F}_2(x)=\pai \frac{\theta}{\theta+x^c}\pad^{\xi}$
 for $\xi, c,\theta>0$. This corresponds to a classical two-sided
jumps risk process with claims given by a Burr distribution with parameters $\xi$, $\theta$ and $c$. Then
$$\chi_{q,\subo}(u)=\lambda_2\sum\limits_{j=1}^{m+1}E(j,0,q)\int_u^\infty e^{-\beta_j(q)(y-u)} \frac{\xi\theta^\xi}{(\theta+y)^{1+\xi}}dy,$$
and applying L'H\^opital's rule twice we get
\begin{align}
 \lim\limits_{u\to\infty}\frac{\colachi_{q,\subo}(u)}{\pai\frac{\theta}{\theta+u^c}
 \pad^{\xi}}&=\lambda_2\sum\limits_{j=1}^{m+1}E(j,0,q)\lim\limits_{u\to\infty}
 \frac{\int_u^\infty e^{\beta_j(q)x}\int_x^\infty
  e^{-\beta_j(q)y} \frac{\xi\theta^\xi}{(\theta+y^c)^{1+\xi}}dy}{\pai\frac{\theta}
  {\theta+u^c}\pad^{\xi}}
=\lambda_2\sum\limits_{j=1}^{m+1}E(j,0,q)\frac{1}{\beta_j(q)}.\nonumber
  \end{align}
Using the second equality in Lemma 5.2 in \cite{kolkovskamartin2}, it follows that
\begin{equation*}
 \lim\limits_{u\to\infty}\frac{\colachi_{q,\subo}(u)}
 {\pai\frac{\theta}{\theta+u^c}\pad^{\xi}}=\lambda_2\constC=
 \frac{\lambda_2a(q)}{q},
\end{equation*}
which implies that  $1-\frac{\colachi_{q,\subo}(u)}{\widehat{\chi}_{q,\subo}(0)}$ is a subexponential distribution. Therefore,
due to Proposition 2 we obtain
$$F_{I_{e_q} }(-u)\approx \frac{\lambda_2}{q}\pai\frac{\theta}{\theta+u^c}\pad^{\xi}\quad\mbox{as}\quad u\to\infty.$$
\textbf{Example 3 (case C)}. Let us suppose that $\subo$ is a spectrally positive $\xi$-stable process,
with $\E\ci\subo(1)\cd=0$ and $\xi \in (1,2)$. In this case
$-\Psi_\subo(r)=r^\xi$ and
$$\chi_{q,\subo}(x)=\frac{1}{\xi} \pai x^{-\xi}-\sum_{j=1}^{m+1}\beta_j(q)E(j,0,q)\int_x^\infty e^{-\beta_j(q)(y-x)}y^{-\xi}dy\pad,$$
where we have used that $E_*(j,0,q)=\beta_j(q)E(j,0,q)$.

Notice that the function $\mathcal{L}_{\xi}(x)=\sum_{j=1}^{m+1}\beta_j(q)E(j,0,q)\int_x^\infty e^{-\beta_j(q)(y-x)}y^{-\xi}dy$
coincides with the function $F_\xi$ defined in \cite{kolkovskamartin3}, hence from Proposition 1 in the aforementioned work we obtain
$$\int_x^\infty\mathcal{L}_{\xi}(y)dy\approx \pai1-\constC\pad\frac{x^{1-\xi}}{\Gamma(2-\xi)}\quad\mbox{as}\quad x\to\infty.$$
Since $\constC=\frac{a(q)}{q}$, it follows that
$\colachi_{q,\subo}(x)\approx C_{\xi} \frac{x^{1-\xi}}{\Gamma(2-\xi)}$ as $x\to\infty,$
where $C_{\xi}=\frac{1}{\xi}\pai\frac{\Gamma(2-\xi)}{\xi-1}-1+\frac{a(q)}{q}\pad$.
Therefore, from Proposition 2 we obtain
$F_{I_{e_q} }(-u)\approx\frac{C_{\xi}}{a(q)\Gamma(2-\xi)}u^{1-{\xi}}$
as $u\to\infty$.

\section{Proof of Lemma \ref{lemmanegativeWHfactor}}\label{proofs}\setcounter{equation}{0}
This section is devoted to the proof of Lemma \ref{lemmanegativeWHfactor}. It requires some preliminary results which we state first.

Let $x_1, x_2,\dots,x_k$ be different complex numbers. For $m\in \{1,2,\ldots\}$ let us denote
$(x_i)_{m}= \underbrace{x_i,x_i,\dots,x_i}_{m \ \text{times}},$ $i=1,\ldots,k$.
The following result follows from standard tools in interpolation theory.
\begin{lemma}\label{diferenciasdivididaslabbe}
 Let $f:\CC\to\CC$ be an analytic function and define $g:\CC^m\to\CC$ by
\begin{equation*}
g[x_1,\dots,x_m]=(-1)^{m-1}\sum\limits_{j=1}^m\frac{f(x_j)}{\prod\limits_{l\neq j}(x_l-x_j)},
\end{equation*} where $x_i \neq x_j$ for $i \neq j.$
Let $n_1$, $n_2,\ldots,n_k$ be given natural numbers, and $m=\sum_{j=1}^kn_j$. 

Then the function $g$ can be extended analytically for multiple points by the expression
 \begin{align}
 g\ci (x_1)_{n_1},\dots,(x_k)_{n_k}\cd&=(-1)^{m-1}\sum\limits_{j=1}^k\frac{(-1)^{m-n_j}}{(n_j-1)!}\frac{\partial ^{n_j-1}}{\partial s^{n_j-1}}\ci\frac{f(s)(x_j-s)^{n_j}}{\prod\limits_{l=1}^k(x_l-s)^{n_l}}\cd_{s=x_j}.\nonumber
 \end{align}
\end{lemma}
We also need the following two lemmas. The first one is
a well-known formula in interpolation theory, while the second one is part of the proof of Proposition 5.4 in \cite{kolkovskamartin2}.
\begin{lemma}Let $m\ge1$ and let $a_1,\ldots, a_m$ be given different numbers. Then
 \begin{equation}\label{formulainterpolacion}
\sum\limits_{j=1}^m\frac{(a_j-s)^k}{\prod\limits_{l=1,l\neq j}^m(a_l-a_j)}=\left\{
\begin{array}{ll}
 (-1)^{m-1},&k=m-1,\\
 0,&k=0,1,\dots,m-2,\\
 \frac{1}{\prod\limits_{j=1}^m(a_j-s)},&k=-1.
\end{array}
\right.
\end{equation}
\end{lemma}

In what follows we set $P_1(r)=\prod_{j=1}^N(\alpha_j-r)^{n_j}$.

\begin{lemma}\label{lemaJ0}
 Let $r,r_1,r_2,\dots,r_{m+1}$ be $m+2$ different complex numbers and define
$$J_0(r)=\lambda_1\sum\limits_{j=1}^{m+1}P_1(r_j)\frac{1}{
 \prod\limits_{k=1,i\neq j}^{m+1}(r_k-r_j)}
  \frac{\frac{Q(-r)}{\prod_{l=1}^N(\alpha_l-r)^{n_l}}-\frac{Q(-r_j)}{\prod_{l=1}^N(\alpha_l-r_j)^{n_l}}}
 {r_j-r}.$$ Then $J_0(r)=0$ for all $r\in \mathbb{C}$.
\end{lemma}

The following result is used in case C.

\begin{lemma}\label{lemmatranslationtailminuslaplaceexponent}
Let $\nu_\mathcal{S}$ be the L\'evy measure of a spectrally positive pure jump
L\'evy process such that
\begin{equation}\label{assumptionnuS}
 \int_{0+}^\infty \pai x^2\wedge x\pad\nu_\mathcal{S}(dx)\menor\infty.
\end{equation}
Then
  \begin{equation}\label{igualdadcolanu}
-\Psi_\subo(r)=r\int_{0+}^\infty\pai1-e^{-rx}\pad\overline{\mathcal{V}}_\subo(x)\,dx.
 \end{equation}
 Moreover, for any $r_1,r_2\in\CC_+$ such that $r_1\neq r_2$ and $\Psi_\mathcal{S}$ exists, there holds\begin{equation}\label{igualdadestraslacion}
 \frac{\Psi_\mathcal{S}(r_1)-\Psi_\mathcal{S}(r_2)}{r_2-r_1}=r_2\widehat{T}_{r_2}\overline{\mathcal{V}}_\subo(r_1)-\frac{\Psi_\subo(r_1)}{r_1}=r_1\widehat{T}_{r_2}\overline{\mathcal{V}}_\subo(r_1)-\frac{\Psi_\subo(r_2)}{r_2}.
\end{equation}
\end{lemma}

\begin{proof}
From assumption  (\ref{assumptionnuS}) we can restrict to $h\equiv1$, and by applying  Fubini's theorem we get (\ref{igualdadcolanu}). On
the other hand, we have:
\begin{align}
 \frac{\Psi_\mathcal{S}(r_1)-\Psi_\mathcal{S}(r_2)}{r_2-r_1}
&=\frac{r_1\frac{\Psi_\mathcal{S}(r_1)}{r_1}-r_2\frac{\Psi_\mathcal{S}(r_2)}{r_2}}{r_2-r_1}=\frac{r_1-r_2}{r_2-r_1}\frac{\Psi_\mathcal{S}(r_1)}{r_1}+r_2\frac{\frac{\Psi_\mathcal{S}(r_1)}{r_1}-\frac{\Psi_\mathcal{S}(r_2)}{r_2}}{r_2-r_1}\nonumber \\
&=r_2\frac{\frac{\Psi_\mathcal{S}(r_1)}{r_1}-\frac{\Psi_\mathcal{S}(r_2)}{r_2}}{r_2-r_1}-\frac{\Psi_\mathcal{S}(r_1)}{r_1}.\label{igualdadPsiS}
 \end{align}
Similarly,
\begin{align}
 \frac{\Psi_\mathcal{S}(r_1)-\Psi_\mathcal{S}(r_2)}{r_2-r_1}&=r_1\frac{\frac{\Psi_\mathcal{S}(r_1)}{r_1}-\frac{\Psi_\mathcal{S}(r_2)}{r_2}}{r_2-r_1}-\frac{\Psi_\mathcal{S}(r_2)}{r_2},\label{igualdadPsiS2}
 \end{align}
and
\begin{align}
  \frac{\Psi_\mathcal{S}(r_1)}{r_1}-\frac{\Psi_\mathcal{S}(r_2)}{r_2}&=\int_{0+}^\infty\ci\frac{1-e^{-r_1x}-r_1x}{r_1}-\frac{1-e^{-r_2x}-r_2x}{r_2}\cd\nu_\mathcal{S}(dx)\nonumber\\
  &=\int _{0+}^\infty\int _0^x \ci e^{-r_1y}-1-\pai e^{-r_2y}-1\pad\cd dy\nu_\mathcal{S}(dx)=\int _{0+}^\infty\int _y^\infty \nu_\mathcal{S}(dx)\ci e^{-r_1y}-e^{-r_2y}\cd dy\nonumber\\
  &=\widehat{\overline{\mathcal{V}}}_\mathcal{S}(r_1)-\widehat{\overline{\mathcal{V}}}_\mathcal{S}(r_2),\nonumber
  \end{align}
 where the third equality follows from Fubini's theorem.
  Substituting the last equality in
  (\ref{igualdadPsiS}) and (\ref{igualdadPsiS2}), and using (\ref{laplaceoperatorT}), we obtain the result.
\end{proof}

Now we are ready to prove Lemma \ref{lemmanegativeWHfactor}.

\begin{proofoflemanegativeWHfactor}
We deal only with the case $q\mayor0$. The case $q=0$
 follows by taking limits when $q\downarrow0$ and that the limit $\lim\limits_{q\downarrow0}\beta_j(q)$ exists and 
$ \lim\limits_{q\downarrow0}\frac{q}{\beta_1(q)}=\E[\x(1)]$.

The proof of the lemma is simpler when the roots $\beta_j(q)$ of the generalized Lundberg equation are simple (see equality (6.11) below). In case of multiple roots we will approximate the Lundberg equation by an Lundberg equation  depending on parameter $\epsilon$ which has  simple roots, and such that when $\epsilon \to 0,$ the roots of the approximating Lundberg equation approximate the multiple roots of the given equation. At the end of the proof we take $\epsilon \to 0$ to obtain Lemma 5 in case of multiple roots $\beta_j(q).$ 

First we obtain the result for the case C. Recall that $\beta_j(q)$ are the roots of the
generalized Lundberg function  of $\x$. Let $\varepsilon\in\pai 0,E\pad$,
where
$$E=\min\left\{\left|Re\pai\beta_i(q)\pad-Re\pai\beta_j(q)\pad\right|: Re\pai\beta_i(q)\pad-Re\pai\beta_j(q)\pad\neq0 \right\},$$
and define the complex numbers
\begin{equation}\label{rhoasterisco}
\begin{array}{l}
\beta_1(q)^*=\beta_1(q),
\beta_2(q)^*=\beta_2(q) + \frac{1}{m+1}\varepsilon,
\dots, \beta_{k_1}(q)^*=\beta_1(q)+\frac{k_1-1}{m+1}\varepsilon,\\ $\,$ \\ \beta_{k_1+1}(q)^*=\beta_2(q),\dots, \beta_{k_1+k_2}(q)^*=\beta_2(q)+\frac{k_1+k_2-1}{m+1}\varepsilon,\\
\beta_{k_1+\dots+k_{R-1}+1}(q)^*=\beta_R(q),\dots, \beta_{m+1}(q)^*=\beta_R(q)+\frac{\sum\limits_jk_j-1}{m+1}\varepsilon,
\end{array}
\end{equation}
where we have omitted the dependence on $\varepsilon$ for simplicity. It follows that 
$\lim\limits_{\varepsilon\rightarrow 0}\beta_{l_j+a_j}(q)^*=\beta_j(q),\ j=1,2,\dots,R,$
for $l_1=0,l_2=k_1,\dots,l_R=k_{R-1}$ and $a_j=1,2,\dots,k_j$.

This gives $m+1$ different numbers $\beta_1^*,\dots,\beta_{m+1}^*$
such that, as $\varepsilon\downarrow0$, the first $k_1$ numbers converge to $\beta_1(q)$,
the next $k_2$ numbers converge to $\beta_2(q)$, and so on.

From the definition of $\Psi_\x$,
\begin{align}
\Psi_\x\ci \beta_j(q)^* \cd-q&=-\Psi_\mathcal{S}\ci \beta_j(q)^* \cd+\lambda_1\frac{Q\pai -\beta_j(q)^*\pad}{\prod_{l=1}^N\pai \alpha_l-\beta_j(q)^*\pad^{n_l}} +c\beta_j(q)^* +\gamma^2\ci\beta_j(q)^* \cd^2-q-\lambda_1. \nonumber
\end{align}
Therefore, for each $j=1,2,\dots,m+1$ we obtain
\begin{align}
\lambda_1+q&=-\Psi_\mathcal{S}\ci \beta_j(q)^* \cd+\lambda_1\frac{Q\pai -\beta_j(q)^*\pad}{\prod_{l=1}^N\pai \alpha_l-\beta_j(q)^*\pad^{n_l}}+c\beta_j(q)^* +\gamma^2\ci\beta_j(q)^* \cd^2-\pai \Psi_\x\ci \beta_j(q)^* \cd-q\pad,\nonumber
\end{align}
which yields
\begin{align}
 \Psi_\x(r)-q&=-\Psi_\mathcal{S}(r)+\lambda_1\frac{Q\pai -r\pad}{\prod_{l=1}^N\pai \alpha_l-r\pad^{n_l}}+cr+\gamma^2 r^2+  \Psi_\mathcal{S}\ci \beta_j(q)^* \cd-\lambda_1\frac{Q\pai -\beta_j(q)^*\pad}{\prod_{l=1}^N\pai \alpha_l-\beta_j(q)^*\pad^{n_l}}\nonumber \\
&-c\beta_j(q)^* -\gamma^2\ci\beta_j(q)^* \cd^2+ \Psi_\x\ci \beta_j(q)^* \cd-q. \label{laplaceconrhoepsilon}
\end{align}

Since $P_1$ is a polynomial with
degree $m$, using Lagrange interpolation we obtain  the equivalent representation
$$P_1(r)=\sum\limits_{l=1}^{m+1}\frac{\prod\limits_{j=1}^N\ci \alpha_j-\beta_l(q)^* \cd^{n_j}}{\prod\limits_{j\neq l}\ci\beta_j(q)^* -\beta_l(q)^* \cd}\prod\limits_{j\neq l}\ci\beta_j(q)^* -r\cd.$$
This and (\ref{laplaceconrhoepsilon}) give
\begin{align}
 \ci \Psi_\x(r)-q\cd P_1(r)
 &=\sum\limits_{l=1}^{m+1}\frac{P_1\ci\beta_l(q)^* \cd}{\prod\limits_{j\neq l}\ci\beta_j(q)^* -\beta_l(q)^* \cd}\prod\limits_{j\neq l}\ci\beta_j(q)^* -r\cd\left\{-(\beta_l(q)^* -r)\frac{\Psi_\mathcal{S}(r)-\Psi_\mathcal{S}(\beta_l(q)^* )}{\beta_l(q)^* -r}\right.\nonumber \\
 &+\left.\phantom{\int_0^{\int}}(\beta_l(q)^* -r)\pai\lambda_1 J_0^*-c-\gamma^2(\beta_l(q)^* +r)\pad+ \Psi_\x(\beta_l(q)^* )-q\right\}\nonumber \\
 &=\prod\limits_{j=1}^{m+1}\ci\beta_j(q)^* -r\cd\sum\limits_{l=1}^{m+1}\frac{P_1\ci\beta_l(q)^* \cd}{\prod\limits_{j\neq l}\ci\beta_j(q)^* -\beta_l(q)^* \cd}\left\{-\frac{\Psi_\mathcal{S}(r)-\Psi_\mathcal{S}(\beta_l(q)^* )}{\beta_l(q)^* -r}\right.\nonumber \\
 &\left.+\lambda_1 J_0^*-c-\gamma^2(\beta_l(q)^* +r)+\frac{ \Psi_\x(\beta_l(q)^* )-q}{\beta_l(q)^* -r}\right\},\label{igualdadconinterpolaciondelagrange}
\end{align}
where $J_0^*=\frac{\frac{Q(-r)}{\prod\limits_{j=1}^N(\alpha_j-r)^{n_j}}-\frac{Q(-\beta_l(q)^*)}{\prod\limits_{j=1}^N(\alpha_j-\beta_l(q)^*)^{n_j}}}{\beta_l(q)^* -r}$.
Formula (\ref{formulainterpolacion}) and Lemma \ref{lemaJ0} imply, respectively:
\begin{equation}\label{suma1}
\begin{array}{ccc}
 \sum\limits_{l=1}^{m+1}\frac{P_1\ci\beta_l(q)^* \cd}{\prod\limits_{j\neq l}\ci\beta_j(q)^* -\beta_l(q)^* \cd}=1&\text{ and }&\lambda_1\sum\limits_{l=1}^{m+1}\frac{P_1\ci\beta_l(q)^* \cd}{\prod\limits_{j\neq l}\ci\beta_j(q)^* -\beta_l(q)^* \cd}J_0^*=0.
\end{array}
\end{equation}
Substituting these two equalities
in (\ref{igualdadconinterpolaciondelagrange}), using the first equality in (\ref{igualdadestraslacion})
to calculate $\frac{\Psi_\mathcal{S}(r)-\Psi_\mathcal{S}(\beta_j(q)^* )}{\beta_j(q)^* -r}$
and dividing by $\prod_{j=1}^{m+1}\ci\beta_j(q)^* -r\cd$, we obtain
\begin{align}
&\ci \Psi_\x(r)-q\cd \frac{P_1(r)}{\prod_{j=1}^{m+1}\ci\beta_j(q)^* -r\cd}\nonumber \\
&=\sum\limits_{l=1}^{m+1}\frac{P_1\ci\beta_l(q)^* \cd}{\prod\limits_{j\neq l}\ci\beta_j(q)^* -\beta_l(q)^* \cd}\Bigg\{-\beta_l(q)^* \widehat{T}_{\beta_l(q)^* }\overline{\mathcal{V}}_\mathcal{S}(r)+\frac{\Psi_\mathcal{S}(r)}{r}-c
-\gamma^2\ci r+\beta_l(q)^* \cd
+\frac{ \Psi_\x\ci\beta_l(q)^* \cd-q}{\beta_j(q)^* -r}\Bigg\}\nonumber \\
&=\frac{\Psi_\mathcal{S}(r)}{r}-c-\gamma^2r-
\sum\limits_{l=1}^{m+1}\frac{P_1\ci\beta_l(q)^* \cd}{\prod\limits_{j\neq l}\ci\beta_j(q)^* -\beta_l(q)^* \cd}\Bigg\{\beta_l(q)^* \widehat{T}_{\beta_l(q)^* }\overline{\mathcal{V}}_\mathcal{S}(r)+\gamma^2\beta_l(q)^*
-\frac{ \Psi_\x\ci\beta_l(q)^* \cd-q}{\beta_j(q)^* -r} \Bigg\}, \label{denominadorcasoC0}
\end{align}
where the last equality follows from (\ref{suma1}).

Using Lemma \ref{lemmakyprianou} and Theorem 2.2 in \cite{lewismordecki} we obtain $\lim\limits_{\varepsilon\downarrow0}\frac{P_1(r)}{\prod_{j=1}^{m+1}\ci\beta_j(q)^* -r\cd}=\pai\kappa(q,-r)\pad^{-1}$. Hence we let $\varepsilon\downarrow0$ in both sides of
(\ref{denominadorcasoC0} and apply Lemma \ref{diferenciasdivididaslabbe}. This yields:
\begin{align}
 &\ci \Psi_\x(r)-q\cd \pai\kappa(q,-r)\pad^{-1}\nonumber \\
&=\frac{\Psi_\mathcal{S}(r)}{r}-c-\gamma^2r-\sum\limits_{l=1}^R\frac{(-1)^{1-k_l}}{(k_l-1)!}\frac{\partial^{k_l-1}}{\partial s^{k_l-1}}\ci\frac{\prod\limits_{j=1}^N(\alpha_j-s)^{n_j}(\beta_l(q)-s)^{k_l}}{\prod\limits_{j=1}^R(\beta_j(q)-s)^{k_j}} s\widehat{T}_s\overline{\mathcal{V}}_\mathcal{S}(r)\cd_{s=\beta_l(q)}\nonumber \\
&-\gamma^2\sum\limits_{l=1}^R\frac{(-1)^{1-k_l}}{(k_l-1)!}\frac{\partial^{k_l-1}}{\partial s^{k_j-1}}\ci\frac{\prod\limits_{j=1}^N(\alpha_j-s)^{n_j}(\beta_l(q)-s)^{k_l}}{\prod\limits_{j=1}^R(\beta_j(q)-s)^{k_j}} s\cd_{s=\beta_l(q)}\nonumber \\
&+\sum\limits_{l=1}^R\frac{(-1)^{1-k_l}}{(k_l-1)!}\frac{\partial^{k_l-1}}{\partial s^{k_l-1}}\ci\frac{\prod\limits_{j=1}^N(\alpha_j-s)^{n_j}(\beta_l(q)-s)^{k_l}}{\prod\limits_{j=1}^R(\beta_j(q)-s)^{k_j}}\frac{ \Psi_\x(s)-q}{s-r}\cd_{s=\beta_l(q)}.\nonumber
\end{align}
Since, for $j=1,2,\dots,R$, $\beta_j(q)$  are roots of
$ \Psi_\x(s)-q=0$ in $\CC_{++}$
with respective multiplicities $k_j,$ it follows from the Leibniz rule
that
$$\sum\limits_{l=1}^R\frac{(-1)^{1-k_l}}{(k_l-1)!}\frac{\partial^{k_l-1}}{\partial s^{k_l-1}}\ci\frac{P_1(s)(\beta_l(q)-s)^{k_l}}{\prod\limits_{j=1}^R(\beta_j(q)-s)^{k_j}}\frac{ \Psi_\x(s)-q}{s-r}\cd_{s=\beta_l(q)}=0.$$
Hence, substituting this in the equality above and setting $D_q=\sum\limits_{l=1}^R\frac{(-1)^{1-k_l}}{(k_l-1)!}\frac{\partial^{k_l-1}}{\partial s^{k_j-1}}\ci\frac{\prod\limits_{j=1}^N(\alpha_j-s)^{n_j}(\beta_l(q)-s)^{k_l}}{\prod\limits_{j=1}^R(\beta_j(q)-s)^{k_j}} s\cd_{s=\beta_l(q)}$,
we obtain:
\begin{align}
&\ci q- \Psi_\x(r)\cd \pai\kappa(q,-r)\pad^{-1}\nonumber \\
&=c+\gamma^2D_q+\gamma^2r-\frac{\Psi_\mathcal{S}(r)}{r}
+\sum\limits_{l=1}^R\frac{(-1)^{1-k_l}}{(k_l-1)!}\frac{\partial^{k_l-1}}{\partial s^{k_l-1}}\ci\frac{\prod\limits_{j=1}^N(\alpha_j-s)^{n_j}(\beta_l(q)-s)^{k_l}}{\prod\limits_{j=1}^R(\beta_j(q)-s)^{k_j}} s\widehat{T}_s\overline{\mathcal{V}}_\mathcal{S}(r)\cd_{s=\beta_l(q)}.\label{denominadorcasoC}
\end{align}
Using Leibniz rule and Lemma \ref{lemaderivadaseintegrales}
we get
\begin{align}
&\sum\limits_{l=1}^R\sum\limits_{a=0}^{k_l-1}\binom{k_l-1}{a}\frac{(-1)^{1-k_l}}{(k_l-1)!}\frac{\partial^{k_l-1-a}}{\partial s^{k_l-1-a}}\ci\frac{\prod\limits_{j=1}^N(\alpha_j-s)^{n_j}(\beta_l(q)-s)^{k_l}}{\prod\limits_{j=1}^R(\beta_j(q)-s)^{k_j}}s\cd_{s=\beta_l(q)}\frac{\partial^a}{\partial s^a} \ci\widehat{T}_s\overline{\mathcal{V}}_\mathcal{S}(r)\cd_{s=\beta_l(q)}\nonumber \\
&=\sum\limits_{l=1}^R\sum\limits_{a=0}^{k_l-1}E_*(l,a,q)\widehat{\mathcal{T}}_{\beta_l(q),a}\overline{\mathcal{V}}_\mathcal{S}(r)=\sum\limits_{l=1}^R\sum\limits_{a=0}^{k_l-1}E_*(l,a,q)\widehat{\mathcal{T}}_{\beta_l(q),a}\overline{\mathcal{V}}_\mathcal{S}(r)
=\widehat{\EL}_q(r),\nonumber
\end{align}
hence  from (\ref{denominadorcasoC}) we obtain
$$\ci q- \Psi_\x(r)\cd\pai\kappa(q,-r)\pad^{-1}=c+\gamma^2D_q+\gamma^2r-\frac{\Psi_\mathcal{S}(r)}{r}+\widehat{\EL}_q(r).$$
Since by L'H\^opital's rule
$\lim\limits_{r\downarrow0}\frac{\Psi_\mathcal{S}(r)}{r}=0$, it
follows that
$a(q)=c+\gamma^2D_q+\widehat{\EL}_q(0)$. Hence
\begin{align}
 \ci q- \Psi_\x(r)\cd\pai\kappa(q,-r)\pad^{-1}
 &=a(q)-\widehat{\EL}_q(0)+\gamma^2r-\frac{\Psi_\mathcal{S}(r)}{r}+\widehat{\EL}_q(r)&\nonumber \\
 &=a(q)+\gamma^2r-\frac{\Psi_\mathcal{S}(r)}{r}-\pai \widehat{\EL}_q(0)-\widehat{\EL}_q(r)\pad,\nonumber
\end{align}
and we obtain the result for case C.

To obtain the result for case B, we use the same notations as above for the $m+1$ roots of the generalized Lundberg equation, and use (\ref{igualdadconinterpolaciondelagrange})
with
$\Psi_\subo(r)$ replaced by $G_\subo(r)$ and setting $\gamma=0$.
This gives:
\begin{align}
&\ci \Psi_\x(r)-q\cd P_1(r)\nonumber \\
&=\prod\limits_{j=1}^{m+1}\ci\beta_j(q)^* -r\cd\sum\limits_{l=1}^{m+1}\frac{P_1\ci\beta_l(q)^* \cd}{\prod\limits_{j\neq l}\ci\beta_j(q)^* -\beta_l(q)^* \cd}\left\{-\frac{G_\subo(r)-G_\subo(\beta_l(q)^* )}{\beta_l(q)^* -r}+\lambda_1 J_0^*-c+\frac{ \Psi_\x(\beta_l(q)^* )-q}{\beta_l(q)^* -r}\right\}\nonumber \\
&=-c+\prod\limits_{j=1}^{m+1}\ci\beta_j(q)^* -r\cd\sum\limits_{l=1}^{m+1}\frac{P_1\ci\beta_l(q)^* \cd}{\prod\limits_{j\neq l}\ci\beta_j(q)^* -\beta_l(q)^* \cd}\left\{\widehat{T}_{\beta_l(q)^* }\nu_\subo(r)+\frac{ \Psi_\x(\beta_l(q)^* )-q}{\beta_l(q)^* -r}\right\},\nonumber
\end{align}
where in the second equality we used (\ref{suma1}) and the fact that
$\frac{G_\mathcal{S}(r)-G_\mathcal{S}(s)}{s-r}=-\widehat{T}_s\nu_\mathcal{S}(r),$
which follows from the second equality in (\ref{laplaceoperatorT}).
Proceeding as in case C, we obtain
\begin{align}
&\ci q- \Psi_\x(r)\cd \pai\kappa(q,-r)\pad^{-1}\nonumber\\
  &=c-\sum\limits_{l=1}^R\sum\limits_{a=0}^{k_l-1}\binom{k_l-1}{a}\frac{(-1)^{1-k_l+a}}{(k_l-1)!}\frac{\partial^{k_l-1-a}}{\partial s^{k_l-1-a}}\ci\frac{\prod\limits_{j=1}^N(\alpha_j-s)^{n_j}(\beta_l(q)-s)^{k_l}}{\prod\limits_{j=1}^R(\beta_j(q)-s)^{k_j}}\cd_{s=\beta_l(q)}\widehat{\mathcal{T}}_{\beta_l(q),a}\nu_\subo(r)\nonumber \\
&=c-\widehat{\ell}_q(r),\nonumber 
\end{align}
Now we set $r=0$ in the above equality to obtain
\begin{equation}\label{igualdadadelta}
 c=a(q)+\widehat{\ell}_q(0).
\end{equation}
This gives the result for case B.

In case A we have
$\Psi_{\x}(r)=\lambda_1\pai\widehat{f}_1(-r)-1\pad-G_\mathcal{S}(r)$. For now we assume that $\subo$ is not a compound Poisson process.
In this case we know from Lemma \ref{lemmarootsofGLEgeneralcase} that the equation
$\Psi_{\x}(r)-q=0$ has only $m$ roots in $\CC_{++}$, denoted as before and such that they have respective multiplicities $k_1\equiv1,k_2,\dots,k_R$, where $\sum_{j=1}^R k_j=m$. Now we consider the numbers $\beta_1^*(q),\dots,\beta_m(q)^*$ as
defined in (\ref{rhoasterisco}) with $\beta_{m+1}(q)^*$ replaced by $\beta_\infty(n)=\sqrt{n}$.
We also take $c_n=\frac{1}{n}$.
In this case, the function $\Psi_\x(r)+c_nr$ is the generalized Lundberg function of some
L\'evy process of the form (\ref{Xalpha}) with $\gamma=0$ and drift term $c_n$. Hence we
can use (\ref{igualdadconinterpolaciondelagrange}) with $\gamma=0$, $G_\subo$ instead of
$\Psi_\subo$ and $c_n$ instead of $c$. This gives
\begin{align}
&\ci \Psi_\x(r)+c_nr-q\cd P_1(r)\nonumber \\
 &=(\beta_\infty(n)-r)\prod\limits_{k=1}^m\ci\rho^*_{k,q}-r\cd\sum\limits_{l=1}^m\frac{P_1\ci\beta_l(q)^*\cd}{(\beta_\infty(n)-\beta_l(q)^*)\prod\limits_{j\neq l}\ci\beta_j(q)^*-\beta_l(q)^*\cd}
 \left\{-\frac{G_\subo(r)-G_\subo(\beta_l(q)^*)}
 {\beta_l(q)^*-r}\right.\nonumber \\
 &\left.+\lambda_1 \frac{\frac{Q(-r)}{\prod_{a=1}^N(\alpha_a-r)^{n_a}}-
 \frac{Q(-\beta_l(q)^*)}
 {\prod_{a=1}^N(\alpha_a-\beta_l(q)^*)^{n_a}}}{\beta_l(q)^*-r}
 -c_n+\frac{\Psi_\x(\beta_l(q)^*)-q}{\beta_l(q)^*-r}
 \right\}\nonumber \\
 &+(\beta_\infty(n)-r)\prod\limits_{k=1}^m\ci\rho^*_{k,q}-r\cd\frac{P_1\ci\beta_\infty(n)\cd}{\prod\limits_{j=1}^m\ci\beta_j(q)^*-\beta_\infty(n)\cd}\Bigg\{-\frac{G_\subo(r)-G_\subo(\beta_\infty(n))}{\beta_l(q)^*-r}\nonumber \\
 &+\lambda_1 \frac{\frac{Q(-r)}{\prod_{a=1}^N(\alpha_a-r)^{n_a}}-\frac{Q(-\beta_\infty(n))}{\prod_{a=1}^N(\alpha_a-\beta_\infty(n))^{n_a}}}{\beta_\infty(n)-r}
 -c_n+\frac{\Psi_\x(\beta_\infty(n))-q}{\beta_\infty(n)-r}
 \Bigg\}.\nonumber
\end{align}
Again, proceeding as in  case C it follows that
 \begin{align}
&\ci \Psi_\x(r)+c_nr-q\cd \frac{P_1(r)}{\prod_{j=1}^m\ci\beta_j(q)^*-r\cd}
=-c_n(\beta_\infty(n)-r)\nonumber \\
&+\sum\limits_{l=1}^m\frac{\beta_\infty(n)-r}{\beta_\infty(n)-\beta_l(q)^*}\frac{P_1\ci\beta_l(q)^*\cd}{\prod\limits_{j\neq l}\ci\beta_j(q)^*-\beta_l(q)^*\cd}\Bigg\{\widehat{T}_{\beta_l(q)^*}\nu_\mathcal{S}(r)+\frac{\Psi_\x\ci\beta_l(q)^*\cd-q}{\beta_l(q)^*-r}\Bigg\}\nonumber \\
&+\frac{P_1\ci\beta_\infty(n)\cd}{\prod\limits_{j=1}^m\ci\beta_j(q)^*-\beta_\infty(n)\cd}\left\{\pai\beta_\infty(n)-r\pad\widehat{T}_{\beta_\infty(n)}\nu_\mathcal{S}(r)+\pai \Psi_\x\ci\beta_\infty(n)\cd-q\pad\right\}\nonumber\\
&=-c_n(\beta_\infty(n)-r)+\sum\limits_{l=1}^m\frac{\beta_\infty(n)-r}{\beta_\infty(n)-\beta_l(q)^*}\frac{P_1\ci\beta_l(q)^*\cd}{\prod\limits_{j\neq l}\ci\beta_j(q)^*-\beta_l(q)^*\cd}\Bigg\{\widehat{T}_{\beta_l(q)^*}\nu_\mathcal{S}(r)
\nonumber \\
&+\frac{\Psi_\x\ci\beta_l(q)^*\cd-q}{\beta_l(q)^*-r}\Bigg\}
+\frac{P_1\ci\beta_\infty(n)\cd}{\prod\limits_{j=1}^m\ci\beta_j(q)^*-\beta_\infty(n)\cd}\Bigg\{-G_\mathcal{S}(r)+G_\mathcal{S}(\beta_\infty(n))
\nonumber \\
&+\pai \Psi_\x\ci\beta_\infty(n)\cd-q\pad\Bigg\},\nonumber
\end{align}
where in the second equality we have used that $(\beta_\infty(n)-r)\widehat{T}_{\beta_\infty(n)}\nu_S(r)=(\beta_\infty(n)-r)\pai\frac{-G_\subo(r)+G_\subo(\beta_\infty(n))}{\beta_\infty(n)-r}\pad$.
Now we substitute $\Psi_\x(\beta_\infty(n))=\lambda_1\pai\frac{Q(-\beta_\infty(n))}{\prod_{a=1}^N(\alpha_a-\beta_\infty(n))^{n_a}}-1\pad+c_n\beta_\infty(n)-G_\mathcal{S}(\beta_\infty(n))$ in the above equality and rearrange terms to obtain
\begin{align}&\ci \Psi_\x(r)+c_nr-q\cd \frac{P_1(r)}{\prod_{j=1}^m\ci\beta_j(q)^*-r\cd}\nonumber \\
&=-c_n(\beta_\infty(n)-r)+\sum\limits_{l=1}^m\frac{\beta_\infty(n)-r}{\beta_\infty(n)-\beta_l(q)^*}\frac{P_1\ci\beta_l(q)^*\cd}{\prod\limits_{j\neq l}\ci\beta_j(q)^*-\beta_l(q)^*\cd}\Bigg\{\widehat{T}_{\beta_l(q)^*}\nu_\mathcal{S}(r) \nonumber \\
&+\frac{\Psi_\x\ci\beta_l(q)^*\cd-q}{\beta_l(q)^*-r}\Bigg\}-\frac{P_1\ci\beta_\infty(n)\cd}{\prod\limits_{j=1}^m\ci\beta_j(q)^*-\beta_\infty(n)\cd}G_\mathcal{S}(r)\nonumber \\
&+\frac{P_1\ci\beta_\infty(n)\cd}{\prod\limits_{j=1}^m\ci\beta_j(q)^*-\beta_\infty(n)\cd}\Bigg\{c_n\beta_\infty(n)+\lambda_1\frac{Q(-\beta_\infty(n))}{\prod_{a=1}^N(\alpha_a-\beta_\infty(n))^{n_a}}-\pai\lambda_1+q\pad\Bigg\} .\label{denominatorXcetagamma0}
\end{align}
Since both polynomials $P_1(r)=\prod_{j=1}^N(\alpha_j-r)^{n_j}$ and $\prod_{j=1}^m\pai\beta_j(q)^*-r\pad$
have degree $m$, and since $Q(r)$ given in (\ref{laplacef1}) is a polynomial of
degree at most $m-1$, it follows that for any fixed $r$ and $s$ such that $r\neq \beta_\infty(n)$ and $s\neq \beta_\infty(n)$,
 $$
  \lim\limits_{n\rightarrow\infty}\frac{\beta_\infty(n)-r}{\beta_\infty(n)-s}=1,\quad \lim\limits_{n\rightarrow\infty}\frac{P_1(\beta_\infty(n))}{\prod_{j=1}^m\pai\beta_j(q)^*-\beta_\infty(n)\pad}=1,\quad \lim\limits_{n\rightarrow\infty}\frac{Q(-\beta_\infty(n))}{\prod_{a=1}^N(\alpha_a-\beta_\infty(n))^{n_a}}=0.
$$
Hence, letting $n\rightarrow\infty$ in both sides of (\ref{denominatorXcetagamma0}) we get
\begin{align}
&\ci q-\Psi_\x(r)\cd\frac{\prod_{j=1}^N(\alpha_j-r)^{n_j}}{\prod_{j=1}^m\pai\beta_j(q)^*-r\pad}\nonumber \\
&=-\sum\limits_{l=1}^m\frac{P_1\ci\beta_l(q)^*\cd}{\prod\limits_{j\neq l}\ci\beta_j(q)^*-\beta_l(q)^*\cd}\Bigg\{\widehat{T}_{\beta_l(q)^*}\nu_\mathcal{S}(r)+\frac{\Psi_\x\ci\beta_l(q)^*\cd-q}{\beta_l(q)^*-r}\Bigg\}+G_\mathcal{S}(r)+\lambda_1+q.\nonumber
\end{align}
The remaining part of the proof is done similarly as in   cases B and C using Lemma \ref{lemmakyprianou}.\hfill
The case when $\subo$ is a compound Poisson process with $\E\ci \subo(1)\cd>0$ is obtained from case B as follows: 

when $c=0$ we have 
$\Psi_\x'(r)|_{r=0+}=\lambda_1\mu_1-\lambda_2\mu_2,$
hence $\psi_\x(r)-q=0$ has $m$ roots under the assumption that $\lambda_1\mu_1-\lambda_2\mu_2>0$.

If $\x_c$ denotes the process $\x$ with drift $c>0$ when $\subo$ is a subordinator and $\x$ denotes the same process with $c=0$ we have 
$\psi_{\x_c}(r)\to\psi_\x(r)$, for all $r>0$ when $c\downarrow0$. 

This implies that the roots of $\psi_{\x_c}(r)-q$ must converge to those of $\psi_{\x_c}(r)-q$, but since the latter function has only $m$ roots, it must hold that one of the roots of $\psi_{\x_c}(r)-q$ (say $\beta_{m+1}(q)$) must converge to infinite when $c\downarrow0$. Hence the result is obtained by the same procedure as before, replacing $c_n$ by $c$ and $\beta_\infty(n)$ by $\beta_{m+1}(q)$.
\end{proofoflemanegativeWHfactor}

{\noindent\bf Acknowledment\ } The authors are grateful to two anonymous referees for their careful reading of the paper and for many useful suggestions which greatly improved the presentation of the results. The first-named author appreciates partial support from CONACyT Grant No. 257867.

\end{document}